\makeatletter

\def\eqalign#1{\,\vcenter{\openup\jot\m@th
  \ialign{\strut\hfil$\displaystyle{##}$&$\displaystyle{{}##}$\hfil
      \crcr#1\crcr}}\,}
\def\eqalignno#1{\displ@y \tabskip\@centering
  \halign to\displaywidth{\hfil$\displaystyle{##}$\tabskip\z@skip
    &$\displaystyle{{}##}$\hfil\tabskip\@centering
    &\llap{$##$}\tabskip\z@skip\crcr
    #1\crcr}}
\def\leqalignno#1{\displ@y \tabskip\@centering
  \halign to\displaywidth{\hfil$\displaystyle{##}$\tabskip\z@skip
    &$\displaystyle{{}##}$\hfil\tabskip\@centering
    &\kern-\displaywidth\rlap{$##$}\tabskip\displaywidth\crcr
    #1\crcr}}

\makeatother

\newdimen\pixel \pixel=.00333333 true in

\documentclass[11pt]{article}
\usepackage{epsfig}
\usepackage{fullpage}

\makeatletter
\def\bigpar{\bigbreak\@afterindentfalse\@afterheading\ignorespaces}
\def\medpar{\medbreak\@afterindentfalse\@afterheading\ignorespaces}
\def\smallpar{\smallbreak\@afterindentfalse\@afterheading\ignorespaces}

\newlength{\saveindent}
\newenvironment{proof}%
      {\bigpar{\bf Proof:}\ 
             \setlength{\saveindent}{\parindent} 
                       \ignorespaces}%
      {\stopproof\ignorespaces\bigbreak \setlength{\parindent}{\saveindent}}

      {\bigpar{\bf Proof:}%
             \setlength{\saveindent}{\parindent} 
                       \,\ignorespaces}%
      {\ignorespaces\bigbreak \setlength{\parindent}{\saveindent}}

      {\bigpar{\bf Proof:}\ %
             \setlength{\saveindent}{\parindent} 
                       \ignorespaces}%
      {\ignorespaces\bigbreak \setlength{\parindent}{\saveindent}}

\newenvironment{proofof}[1]%
      {\bigpar{\bf#1:}\ %
             \setlength{\saveindent}{\parindent} 
                       \ignorespaces}%
      {\stopproof\ignorespaces\bigbreak \setlength{\parindent}{\saveindent}}

\newenvironment{remark}%
      {\smallpar{\bf Remark:}\ 
                       \ignorespaces}%
      {\stopproof\ignorespaces\medbreak \setlength{\parindent}{\saveindent}}

\newenvironment{remark*}%
      {\smallpar{\bf Remark:}\ 
                       \ignorespaces}%
      {\ignorespaces\medbreak \setlength{\parindent}{\saveindent}}

      {\smallpar{\bf Remarks:}\ 
                       \ignorespaces}%
      {\stopproof\ignorespaces\medbreak \setlength{\parindent}{\saveindent}}

\newenvironment{remarks*}%
      {\smallpar{\bf Remarks:}\ 
                       \ignorespaces}%
      {\ignorespaces\medbreak \setlength{\parindent}{\saveindent}}

      {\smallpar{\bf Remark:}\ %
                       \ignorespaces}%
      {\stopproof\ignorespaces\medbreak \setlength{\parindent}{\saveindent}}

      {\smallpar{\bf Remarks:}\ %
                       \ignorespaces}%
      {\stopproof\ignorespaces\medbreak \setlength{\parindent}{\saveindent}}
\makeatother

\newtheorem{theorem}{Theorem}
\newtheorem{lemma}[theorem]{Lemma}
\newtheorem{claim}[theorem]{Claim}
\newtheorem{proposition}[theorem]{Proposition}

\newtheorem{definition}[theorem]{Definition}
\newtheorem{example}{Example}

\def\begex{\begin{example}\parindent=0pt \rm}
\def\endex{\end{example}}
\def\square{\vbox{\hrule height.2pt\hbox{\vrule width.2pt height5pt \kern5pt
                                   \vrule width.2pt} \hrule height.2pt}}
\def\stopproof{\hfill \square \smallskip}

\def \lf {{\lfloor}}
\def \rf {{\rfloor}}
\def \A {{ ({t \over n} \wedge 1) }}

\def\sfrac#1#2{{\textstyle{#1 \over #2}}}
\def\floor#1{{{\lf #1 \rf}}}
\def\ceiling#1{{{\lceil #1 \rceil}}}
\def \lc {{\lceil}}

\def \rc {{\rceil}}

\def \sgeq {{\succeq}}

\def\half{{\textstyle{1\over2}}}
\def \matches{{\cal M}}
\def\quarter{{\textstyle{1\over4}}}

\def\third{{\textstyle{1\over3}}}

\def\eighth{{\textstyle{1\over8}}}

\def \r {{\bf R}}

\def \r {{ \cal F}}
\def \ms {{ m}}
\def \mss {{ m^*}}

\def\r|{{\Bigr\vert}}
\def\l|{{\Bigl\vert}}

\def \R {{\bf R}}

\def\phi {\Phi}

\def\e{\epsilon}
\def\twosubs#1#2{\scriptstyle#1\atop\scriptstyle#2}

\def \pt {{\tilde p}}

\def\varepsilon{\mathchar"122 }

\def \one {{\mathbf 1}}
\def \chi {{\mathbf 1}}

\def\law{{\cal L}}

\def\p {{ \bf P}}
\def\P {{ \bf P}}

\def\e {{ \bf {E}}}

\def \id {{\rm id}}

\def \sgood {{0}}
\def \sbad {{ \infty}}
\def \wbad {{ \sbad}}
\def \wgood {{ \sgood}}

\def\Sq{{\cal S}_q}
\def\Sq-{{\cal S}_{q-1}}

\def\u{{\cal U   }}

\def \S {{\cal S}}

\def \mt {{\widehat \mu}}

\def \lt {{\widehat \lambda}}

\def \Z {\bf Z}

\def \kss {{ k^*}}

\def\f {{\cal F}}

\def\given {{\,|\,}}
\def\biggiven {{\,\Bigl|\,}}
\def\one{{\mathbf 1}}

\def\ent{{ \rm \sc ENT}}
\def \phat {{ \widehat P}}

\def\enthat{{ \rm {{\sc ENT}}}}

\def\pm{{ \pi_{(m)}}}
\def\pt{{ \pi_{(t)}}}

\def \fkp {{ {\widehat \f_{k+1}}}}

\newcommand{\be}{\begin{equation}}
\newcommand{\ee}{\end{equation}}
\newcommand{\lab}{\label}
\begin{document}
\title{Improved mixing time bounds for the \\ 
  Thorp shuffle
and
$L$-reversal chain
}
\author{
{\sc Ben Morris}\thanks{Department of Mathematics,
University of California, Davis.
Email:
{\tt morris@math.ucdavis.edu}.
Research partially supported by Sloan Fellowship and
NSF grant DMS-0707144.}
}  
\date{}
\maketitle
\begin{abstract}
\noindent 
We prove a theorem 
that reduces bounding the mixing time 
of a card shuffle to 
verifying a 
condition that involves only pairs of cards,
then  we use it to obtain improved bounds for
two previously studied models. 
 
E.~Thorp  
introduced the following card
shuffling model 
in 1973. 
Suppose the number of cards $n$ is even. 
Cut the deck into two equal piles. 
Drop the first card from the left pile or from the right pile
according 
to the outcome of a fair coin flip. Then drop from the other pile. 
Continue this way until both piles are empty. We obtain a mixing 
time bound of $O(\log^4 n)$. Previously, the best known bound 
was $O(\log^{29} n)$ and previous proofs were only valid for $n$ 
a power of $2$. 

  We also analyze the following model, 
called the  {\it $L$-reversal chain}, introduced by Durrett. 
There are $n$ cards arrayed in 
a circle. 
Each step, an interval of cards of length at most $L$
is chosen uniformly at random and its order is reversed.
Durrett has conjectured that the mixing time is 
$O(\max(n, {n^3 \over L^3})\log n)$. 
We obtain a bound that is 
within a factor $O(\log^2 n)$ of this, the first bound within 
a poly log factor of the conjecture.

\end{abstract}
\setcounter{page}{1}

\section{Introduction} \lab{intro}
Card shuffling has a rich history in mathematics, dating back to 
work of Markov \cite{markov} and Poincare
\cite{poincare}. A basic problem is to determine 
the mixing time, i.e., the number of shuffles necessary to mix up the deck
(sec Section \ref{secaps} for a precise definition).
A natural first step (used as far back
Borel and Cheron \cite{bc} in 1940)
is to determinine the number of steps necessary to 
randomize single cards and pairs. 
Clearly this is always a lower bound for the mixing time.
On the other hand, it is often not far from  an upper bound
as well; for a number of models of card shuffling
(see, e.g., Diaconis and Shahshahani \cite{ds},  
Wilson \cite{wilson}, or Bayer and Diaconis \cite{bd})
the 
the mixing time is only a small factor (e.g. $O(1)$ or $O(\log n)$)
larger than the time required to mix pairs. 
This suggests finding a general method 
that reduces bounding the mixing time (in the global sense
that the distribution on all $n!$ permutations is roughly uniform) to 
verifying a local
condition that involves only pairs of cards. 
In this paper, we introduce such a method 
and use it to analyze two previously studied models.
In both cases we find an upper bound for the mixing time that 
is within a poly logarithmic factor of optimal. 

 We study card shuffles that can be viewed as generalizations
of three card Monte. In three card Monte, the 
cards are spread out face down on a table. 
In one step, the dealer chooses two cards, puts them together and then 
separates them quickly so that an observer cannot tell which is which. 
We call this operation a {\it collision}, and 
model it mathematically as a random permutation that is an even mixture 
of a transposition and the identity. 
We prove a general theorem 
that applies to any method of shuffling that uses collisions. The theorem 
bounds the change in relative entropy after many steps of the chain,
based on something that is related to the interactions between 
pairs of cards. 
Next we use the theorem to analyze two card shuffling models,
the Thorp shuffle and Durrett's $L$-reversal
model. 

\subsection{Applications}
\label{secaps}
 In this section we describe two applications of our main 
theorem. 
First, we give a formal definition of the mixing time.
Let $p(x,y)$ 
be transition probabilities for a   Markov chain on a finite state space $V$
with a uniform stationary distribution.
For probability measures $\mu$ and $\nu$ on $V$,
define the total variation distance
$|| \mu - \nu || = \sum_{x \in V} |\mu(x) - \nu(x) |$, 
and 
define the mixing time
\be
\label{mixingtime}
T_{\rm mix} = \min \{n: || p^n(x, \, \cdot) - \u || \leq \quarter \mbox{ for all $x \in V$}\} \,,
\ee
where $\u$ denotes the uniform distribution.

 Our first application is the Thorp shuffle,
which is defined as follows. 
Assume that the number of cards, $n$, is even.
Cut the deck into two equal piles. 
Drop the first card from the left pile or the
right
pile according to the outcome of a fair coin flip;  
then drop from the other pile. 
Continue this way, 
with independent coin flips deciding whether to drop {\sc left-right} or 
{\sc right-left} each time, until 
both piles are empty. 

 The  Thorp shuffle, despite its simple description, 
has been hard
to analyze. Determining its mixing 
time 
has been called the ``longest-standing open card shuffling problem'' \cite{per}. 
In \cite{thorp} the author obtained the first 
poly log upper bound, proving a 
bound of $O(\log^{44} n)$, 
valid when $n$ is a power of $2$.  
Montenegro and Tetali \cite{mt-thorp} built on this
to get a bound of $O(\log^{29} n)$. 
In the present paper, we dispense with 
the power-of-two assumption
and get an improved bound of $O(\log^4 n)$.

We also analyze a  Markov chain that was introduced by 
Durrett \cite{durrett} as a model for evolution of a 
genome (see \cite{durrettbio}). 
In the {\it $L$-reversal chain} there are two parameters, $n$ and $L$.
The cards are located at the
vertices of an $n$-cycle, which we label
$0, \dots, n - 1$. 
Each step, a (nonempty) interval of cards of length at most $L$
is chosen uniformly at random and its order is reversed.
By the coupon collector problem, $O(n \log n)$ steps are needed 
to break adjacencies between neighboring pairs. 
Furthermore, the mixing time for a single card is on the order
${n^3 \over L^3}$, because each step 
the probability that a particular card moves is on the order of
$L/n$ and each time a card moves it performs a step of
a symmetric random walk with 
typical displacement on the order $L$. 
These considerations led Durrett to the 
following conjecture. \\
\\
{\bf Conjecture (Durrett). } The mixing time for the $L$-reversal
chain is $O(\max(n, {n^3 \over L^3})\log n)$. \\
\\
In \cite{durrett}, Durrett proves the corresponding 
lower bound using Wilson's technique
\cite{wilson}
based  
on eigenfunctions. 
The 
spectral gap was determined 
to be within constant factors of $\max(n, {n^3 \over L^3})$ by 
Cancrini, Caputo and Martinelli
\cite{ccm}. 
The best previously-known bound for the mixing time, 
which could be obtained by applying standard comparison techniques,  
was within a factor $O(n^{2/3})$ 
of the Durrett's conjecture in the worst case. 

 Durrett's conjecture
has presented a challenge to 
existing techniques. 
As shown by Martinelli et al, 
the log Sobolev constant does not give the conjectured 
mixing time. Furthermore, the mixing time in $L^2$ 
(defined by replacing total variation distance by 
an appropriate $L^2$ distance in equation (\ref{mixingtime}))
can be 
nearly $n^{1/3}$ times the 
conjecture, as the following example shows. 
Let $L = n^{2/3}$, so that the conjectured mixing time is $O(n \log n)$. 
We claim that in this case the $L^2$ mixing time is at least $c n^{4/3}$ for a 
constant $c$. 
Let
$A$ be the event that cards $1, \dots, n/2$ occupy positions 
$1, \dots n/2$ in any order.
If the initial ordering is the identity permutation, 
then after $t$ shuffles we have
\begin{eqnarray*}
\P(A) &\geq& \P(\mbox{none of the reversed intervals contained cards $1$ or $n/2$}) \\
      &\geq& \Bigl(1 - {2L \over n} \Bigr)^t,
\end{eqnarray*}
which is much larger than ${n \choose n/2}^{-1}$ unless  $t \geq c n^{4/3}$ for a constant $c$. 
Since mixing in $L^2$ implies 
convergence of transition probabilities,
the $L^2$ mixing time is at least 
on the order of $n^{4/3}$, which is
higher than the conjecture. This means that in order to 
prove the conjectured bound on the mixing time in total variation, one cannot use any method  
for bounding mixing times that gives a bound in $L^2$.

In the present paper, 
we prove that the mixing time is  $O\Bigl( 
(n \vee {n^3 \over L^3})\log^3n  \Bigr)$. This is 
the first upper bound 
that is within a poly log factor 
of the conjecture.

 The remainder of this paper is organized as follows. In Section 
\ref{bk} we give some necessary background on entropy and prove some
elementary inequalities. In Section \ref{gen} 
we define {\it Monte shuffles}, the general model of card shuffling
to which our main theorem will apply. In Section \ref{secmain} we prove the main theorem.
In Section \ref{secthorp} we analyze the Thorp shuffle
and 
in Section \ref{seclrev} we analyze the $L$-reversal chain.
\section{Background} 
\label{bk}
 For a
probability distribution $\{p_i: i \in V\}$, define the 
(relative) entropy of $p$ by 
$\ent(p) = \sum_{i \in V} p_i \log (|V| p_i)$,
where we define $0 \log 0 = 0$.
The following well-known inequality links relative entropy to total variation distance.
Let $\u$ denote the uniform distribution over $V$. Then
\begin{equation}
\label{totent}
|| p - \u || \leq \sqrt{ \half \ent(p)}.
\end{equation}
If $X$
is a random variable (or random permutation) 
taking finitely many values, 
define $\ent(X)$ as the relative entropy of the
distribution of $X$.
Note that if $\P(X = i) = p_i$  for $i \in V$ then
$\ent(X) = \e(\log (|V| p_X))$. 
We
shall think
of the distribution 
of a random permutation
in $\S_n$ as a sequence of probabilities of length $n!$, 
indexed by permutations in $\S_n$.
If $\f$ is a sigma-field, then we shall write 
$\ent(X \given \f)$ for the relative entropy of the conditional
 distribution of $X$ given $\f$. 
Note that $\ent(X \given \f)$ is a random variable.
If $\pi$ is a random permutation
in $S_n$, 
then for $1 \leq k \leq n$,  
define
$\f_k = \sigma( \pi^{-1}(k), \dots, \pi^{-1}(n))$, 
and 
define $\ent(\pi, k) = \ent( \pi^{-1}(k) \given \f_{k+1})$
(where we think of the 
conditional distribution 
of $\pi^{-1}(k)$ given $\f_{k+1}$ 
as being a sequence of length
$k$). 
The standard entropy chain rule (see, e.g., \cite{cover})
gives the following proposition.
\begin{proposition}
\label{decomp}
For any $i \leq n$ we have
\[
\ent(\pi) = \e \Bigl(\ent( \pi \given \f_{i})\Bigr) 
+ \sum_{k=i}^n \e(\ent(\pi, k)).
\]
\end{proposition}
To compute the relative entropy in first term on the right hand side,
we think of the distribution of $\pi$ given $\f_i$ as a sequence of
probabilities of length $(i-1)!$. 
\begin{remark}
Substituting $i = 1$ into the formula 
gives $\ent(\pi) = \sum_{k=1}^n \e(\ent(\pi,k))$.
\end{remark} 
If we think of $\pi$ as representing the order of a deck of cards,
with $\pi(i) = \mbox{location of card $i$}$, then this allows us to 
think of $\e(\ent(\pi, k))$ as the portion of the overall entropy
$\ent(\pi)$
that is attributable to the location $k$. 
If $S \subset \{1, \dots, n\}$ is a set of positions then 
we shall refer to the quantity $\sum_{k \in S} \ent(\pi, k)$
as the entropy that is {\it attributable to $S$}.

\begin{definition}
For $p, q \geq 0$, define $d(p, q) = \half p\log p + \half q \log q - 
{p + q \over 2} \log\Bigl({p + q \over 2} \Bigr)$.
\end{definition}
We will need the following proposition.
\begin{proposition}
\label{conv}
Fix $p \geq 0$. 
The function $d(p,\,\cdot)$ is convex.
\end{proposition}
\begin{proof}
A calculation shows that the second derivative is positive.
\end{proof}
Observe that $d(p,q) \geq 0$, with equality iff $p=q$ by the 
strict convexity
of the function $x \to x \log x$. Furthermore, some calculations 
give
\be
\lab{eq-df}
d(p, q) = {p + q \over 2} f\Bigl( {p-q \over p+q} \Bigr),
\ee
where $f(\Delta) = 
\half (1+ \Delta) \log(1 + \Delta) +
\half (1 - \Delta) \log(1 - \Delta)$. 
If $p = \{p_i: i \in V\}$ and $q = \{q_i: i \in V\}$ 
are both probability distributions on $V$, then we can define the 
``distance'' $d(p,q)$ between $p$ and $q$, by $d(p,q) = \sum_{i \in V}
d(p_i, q_i)$. 
(We use the term {\it distance} loosely and don't claim that $d(\cdot,\,
\cdot)$ satisfies the triangle inequality.)
Note that $d(p,q)$ is the difference between the average
of the entropies of $p$ and $q$ and the entropy of the average
(i.e. an even mixture) of $p$ and $q$. 

We will use the following projection lemma.
\begin{lemma}
\label{projection}
Let $X$ and $Y$ be random variables with distributions $p$ and $q$,
respectively. Fix a function $g$ and let $P$ and $Q$ be the distributions of $g(X)$ and
$g(Y)$, respectively. Then $d(p, q) \geq d(P,Q)$. 
\end{lemma}
\begin{proof}
Let $S_i = \{x: g(x) = i\}$. Then
\[
P_i = \sum_{x \in S_i} p_x; 
\hspace{.4 in}
Q_i = \sum_{x \in S_i} q_x.
\]
We have
\begin{eqnarray}
d(p, q) 
&=& \sum_i \sum_{x \in S_i}
d(p_x, q_x) \\
&=& \sum_i \sum_{x \in S_i}
{p_x + q_x \over 2} f\Bigl( {p_x-q_x \over p_x+q_x} \Bigr) \\
\label{jen}
&=& \sum_i 
\Bigl[ {P_i + Q_i \over 2}\Bigr]
\sum_{x \in S_i}
{p_x + q_x \over 2} 
\Bigl[ {P_i + Q_i \over 2}\Bigr]^{-1}
f\Bigl( {p_x-q_x \over p_x+q_x} \Bigr).
\end{eqnarray}
Note that $f$ has a positive second derivative, hence is 
is convex. Thus by Jensen's inequality, the quantity (\ref{jen}) 
is at least
\begin{eqnarray}
\sum_i 
\Bigl[ { P_i + Q_i \over 2}
\Bigr]
f\Bigl( \sum_{x \in S_i}
{p_x + q_x \over 2} 
\Bigl[ {P_i + Q_i \over 2}\Bigr]^{-1}
{p_x-q_x \over p_x+q_x} \Bigr) &=&
\sum_i 
\Bigl[ { P_i + Q_i \over 2}
\Bigr]
f\Bigl( 
{P_i - Q_i \over P_i + Q_i} \Bigr) \\
&=& \sum_i d(P_i, Q_i) \\
&=& d(P, Q).
\end{eqnarray}
\end{proof}
Let $\u$ denote the uniform distribution on $V$. Note that if $\mu$ is 
an arbitrary distribution on $V$, then $\ent(\mu)$ and 
$d(\mu, \u)$ are both notions of a distance from $\mu$ to $\u$.
The following lemma relates the two.
\begin{lemma}
\label{dent}
For any distribution $\mu$ on $V$ we have
\[
d(\mu, \u) \geq {c \over \log |V|} \ent(\mu),
\]
for a universal constant $c > 0$.
\end{lemma}
\begin{proof}
Let $n = |V|$, 
define $\mt = n\mu$ and 
define $g:(0, \infty) \to \R$ by $g(x) = x \log x - (x - 1).$ Then
\begin{eqnarray}
\ent(\mu) &=& \sum_{i \in V} \mu(i) \log(n \mu(i)) \\
&=& {1 \over n} \sum_{i \in V} \mt(i) \log \mt(i) - (\mt(i) - 1) \\
&=& {1 \over n} \sum_{i \in V} g(\mt(i)),
\end{eqnarray}
where the second equality holds because $\sum_{i \in V} (\mt(i) - 1) = 0$.
Thus it's enough to show for a universal constant $c$ we have
\begin{equation}
\label{showd}
d(\mu(i), \sfrac{1}{n}) \geq {c \over n \log n} g( \mt(i)),
\end{equation}
for all $i \in V$. Fix $i \in V$ and let $x = \mt(i)$. Then
by equation (\ref{eq-df}) we have
\begin{eqnarray}
d( \mu(i), \sfrac{1}{n}) &=& {1 \over n} d(x, 1) \\
&=& {1 \over n} \Bigl( { x  + 1 \over 2} \Bigr) f\Bigl( {x -1 \over x+1} \Bigr),
\end{eqnarray}
where $f(\Delta) = 
\half (1+ \Delta) \log(1 + \Delta) +
\half (1 - \Delta) \log(1 - \Delta)$. 
Thus it remains to show that the function $R(x)$ defined by
\be
\lab{ratio}
R(x) = 
{g(x) \over \Bigl( { x  + 1 \over 2} \Bigr) f\Bigl( {x -1 \over x+1} \Bigr)}
\ee
is at most $c^{-1} \log n$ on the interval $[0,n]$, for a constant $c>0$.
Note that $R(x)$ is bounded on the interval $[0,2]$.
(This can be seen by applying L'Hopital's rule twice for the point $x = 1$.) Let $x \in [2,n]$.
The denominator in (\ref{ratio}) is at least 
\[
{x \over 2} f\Bigl( {x - 1 \over x+1} \Bigr) \geq {x \over 2} f(\third),
\]
since the function $x \to f(x-1/x+1)$ is increasing
on $[2, \infty)$. The  
numerator is $g(x) \leq x \log x \leq x \log n$.
Thus $R(x) \leq {2\log n/f(\third)}$ on the interval $[2,n]$
and the proof is complete.
\end{proof}
\section{General set-up: card shuffles with collisions}
\label{gen}
\subsection{Collisions}
\label{mset}
We shall now define a {\it collision}, 
which is the basic ingredient in all of the card
shuffles analyzed in the present paper. 
If $\pi$ is a random permutation in $S_n$ such that
\[
\pi= 
\left\{\begin{array}{ll}
\id & 
\mbox{with probability $\half$;}\\
(a,b) & \mbox{with probability $\half$,}\\
\end{array}
\right.
\]
for some $a,b \in \{1, 2, \dots, n\}$ (where
we write $\id$ for  the identity permutation and 
$(a,b)$ for the transposition of $a$ and $b$),
then we will call $\pi$ a {\it
collision.} If $\pi$ and $\mu$ are permutations in $S_n$, then 
we write $\pi \mu$ for the composition $\mu \circ \pi$. 

A card shuffle can be described as 
a
random permutation chosen 
from 
a certain 
probability distribution.
If we start with 
the identity permutation and
each shuffle has the distribution of
$\pi$, then after $t$ steps the cards are distributed
like
$\pi_1 \cdots \pi_t$,
where the $\pi_i$ are i.i.d.~copies of $\pi$. 
In this paper, we shall consider shuffling permutations $\pi$ that can be
written in the form
\be
\lab{form}
\pi = \nu c(a_1, b_1) c(a_2, b_2) \cdots c(a_k, b_k),
\ee
where $\nu$ is an arbitrary random permutation, 
the numbers $a_1, \dots, a_k, b_1, \dots, b_k$ are disjoint, and 
$c(a_j, b_j)$ is a collision of $a_j$ and $b_j$.
The values of $a_j$ and $b_j$ and the number of collisions (which can
be zero) may depend on $\nu$, 
but conditional on $\nu$ the $c(a_j, b_j)$ are independent 
collisions.  We shall call shuffles of this type {\it Monte}. 

For $t \geq 1$, define $\pi_{(t)} = \pi_1 \cdots \pi_t$.
\subsection{Warm-Up Lemma}
\label{abuse}
In this section we prove a simple lemma with a short proof that brings
out many of the central ideas of our main theorem (Theorem
\ref{maintheorem} below). 
We start with an easy proposition.
\begin{proposition}
\label{det}
Suppose that $\pi$ is any fixed permutation. Then
\[
\ent(\mu \pi) = \ent(\mu).
\]
\end{proposition}
\begin{proof}
Up to a re-labeling of indices, the random permutation $\mu \pi$ 
has the same distribution as $\mu$, hence the same relative entropy.
\end{proof}
If $\pi$ is random and independent of $\mu$ then 
$\ent(\mu \pi) \leq \ent(\mu)$, which follows by conditioning on
$\pi$,
applying Proposition \ref{det}, 
and then applying Jensen's inequality to the 
function $x \to x \log x$. It follows that if $\pi_1,
\pi_2, \dots$ are i.i.d.~copies of $\pi$ then $\ent(\pi_1 \cdots
\pi_k)$ 
is nonincreasing in $k$. In this section we study the decay of entropy 
$\ent(\mu \pi) - \ent(\mu)$ in the case where the permutation $\pi$ is
a collision.  

The following lemma relates to the case where $\pi$ is a 
collision 
between the $j$th
card 
and another card of smaller index. The lemma says that 
the relative entropy is reduced by at least $c \ent(\mu, j)/\log n$, on 
average (where ``on average'' means with respect to the different
possible choices of indices $i\leq j$). 
\begin{lemma}
\label{mainlemma}
Let $\mu$ be a random permutation.
Then for a universal constant $c$ we have
\[
j^{-1} \sum_{i \leq j} \ent(\mu c(i, j))
\leq \ent(\mu) - c \ent(\mu, j)/\log n.
\]
\end{lemma}
\begin{proof}
Using the abuse of notation 
$\half \pi_1 + \half \pi_2$ for a random permutation whose
distribution is an even mixture of the distributions of $\pi_1$ and
$\pi_2$, 
we have
\[
\mu c(i, j) 
= \half \mu + \half 
\mu (i, j).
\]
Let $\law(X \given \f)$ denote the conditional 
distribution of random variable (or random
permutation)  $X$ given the sigma field $\f$.
Let $\mt = \mu(i,j)$ (i.e., the product of $\mu$ and the transposition $(i,j)$).
Note that $\mt$ and $\mu$ are the same, except that
$\mt^{-1}(i) = \mu^{-1}(j)$ and $\mu^{-1}(i) = \mt^{-1}(j)$
and 
recall that $i \leq j$. 
It follows that $\ent( 
\mt \given \f_{j+1}) = \ent(\mu \given \f_{j+1})$ 
and hence
$
\ent ( 
\mu
c(i,j) \given \f_{j+1}
) - \ent(\mu \given \f_{j+1}) = 
-d(
\law(
\mt \given \f_{j+1}),
\law(\mu \given \f_{j+1}))$. But by the projection lemma, 
\begin{eqnarray*}
d(
\law(
\mt \given \f_{j+1}),
\law(\mu \given \f_{j+1})) &\geq&
d(\law(\mt^{-1}(j) \given \f_{j+1}),
\law(\mu^{-1}(j) \given \f_{j+1})) \\
&=&
d( \law(\mu^{-1}(i) | \f_{j+1}), \law(\mu^{-1}(j) \given \f_{j+1})).
\end{eqnarray*}
Hence
\begin{eqnarray}
\nonumber
j^{-1} \sum_{i \leq j} \ent(\mu c(i, j)  \given \f_{j+1}) - \ent(\mu \given
\f_{j+1}) 
&=&
-j^{-1} \sum_{i \leq j} d(\law(\mu^{-1}(i) \given \f_{j+1}), 
\law(\mu^{-1}(j) \given \f_{j+1})) \\
\nonumber 
&\leq&
-d\Bigl( j^{-1}\sum_{i \leq j} \law(\mu^{-1}(i) \given \f_{j+1}), 
\law(\mu^{-1}(j) \given \f_{j+1}) \Bigr)
  \\
\nonumber 
&=&
-d\Bigl(\u, \law(\mu^{-1}(j) \given \f_{j+1})\Bigr) \\
&\leq& 
-{c \over \log n} \ent( \law(\mu^{-1}(j) \given \f_{j+1}),
\end{eqnarray}
where the first inequality is by Proposition \ref{conv} and the second is 
by Lemma \ref{dent}. 
Here $\u$ denotes the uniform distribution over 
$\{1, \dots, n\} - \{ \mu^{-1}(j+1), \dots, \mu^{-1}(n) \}$.
Taking expectations gives 
\begin{eqnarray}
\label{diff}
j^{-1} \sum_{i \leq j} \e(\ent(\mu c(i, j) \given \f_{j+1} )) - \e(\ent(\mu \given \f_{j+1})) 
&\leq&
-{c \over \log n} \ent( \mu , j).
\end{eqnarray}

Since $\ent(\mu, k) 
= \ent(\mu c(i,j), k)$ for all $k \geq j+1$, Proposition
\ref{decomp}
and equation (\ref{diff}) yield the lemma.
\end{proof}
\section{Main Theorem}
\label{secmain}
Let $\pi$ be a random permutation in $\S_n$ that is Monte 
(i.e., can be written in the form
(\ref{form})) and let $\pi_1, \pi_2, \dots$ be independent copies of
$\pi$. For $t \geq 1$
let $\pt = \pi_1 \cdots \pi_t$. \\
\\
{\bf Convention. }
We shall use the following convention throughout.
For integers $x$ with $1 \leq x \leq n$, we denote by {\it card $x$} 
the card initially in position $x$. \\
\\
For cards
$x$ and $y$, say that {\it $x$ collides with $y$ at time $m$}
if for some $i$ and $j$ we have
$\pm^{-1}(i) = x$, $\pm^{-1}(j) = y$, and 
$\pi_m$
has a collision of $i$ and $j$.

We will need the following definition.
\begin{definition}
  For a random variable $X$, 
a finite set $S$ and a real number $A \in [0,1]$, 
say that the distribution of $X$
is $A$-uniform over $S$ if 
\[
\P(X = i) \geq A|S|^{-1},
\]
for all $i \in S$. 
\end{definition}
\begin{remark}
If
$A < 1$ then the distribution of $X$ need not be concentrated on $S$. 
(But if $A = 1$, then $X$ is uniform over $S$.) 
\end{remark}

Our main theorem
is a generalization of Lemma 
\ref{mainlemma}. It generalizes from a collision to an arbitrary 
Monte shuffle, and it
bounds the loss in relative entropy after many
steps.
\begin{theorem}
\label{maintheorem}
Let $\pi$ be a Monte shuffle on $n$ cards. 
Fix an integer $t > 0$
and suppose that $T$ is a random variable taking values in $\{1, \dots, t\}$,
which is independent of the shuffles $\{\pi_i: i \geq 0\}$. 
For a card $x$, let $b(x)$
denote the first card 
to collide with $x$ after time $T$ (or
$b(x) = x$ if there is no such card). Define the match $m(x)$ 
of $x$ by
\[
m(x) := 
\left\{\begin{array}{ll}
b(x) & 
\mbox{if $x = b(b(x))$;}\\
x & \mbox{otherwise.}\\
\end{array}
\right.
\]
Suppose that for every card $i$ there is a constant $A_i \in [0,1]$ such 
that the distribution of $m(i)$ is $A_i$-uniform over $\{1, \dots, i\}$.
Let $\mu$ be an arbitrary random permutation that
is independent of $\{\pi_i : i \geq 0\}$.
Then 
\[
\ent(\mu \pt) - \ent(\mu) \leq {-C \over \log n} 
\sum_{k=1}^n A_k E_k,
\]
where $E_k = \e( \ent(\mu, k))$
and
$C$ is a universal constant.
\end{theorem}
\begin{proof}
Let $\matches = (m(i): 1 \leq i \leq n)$. 
For $i$ and $j$ with $j \leq i$, let
$c(i, j)$ be a collision of $i$ and $j$. Assume that all 
of the $c(i,j)$ are independent of $\mu$,
$\pt$ and each other. Note that
\[
\Bigl[ \prod_{i: m(i) \leq i} c(i, m(i)) \Bigr] \pt
\]
has the same distribution as $\pt$, so it is enough to bound the 
relative
entropy of the distribution of $\mu 
\Bigl[ \prod_{i: m(i) \leq i} c(i, m(i)) \Bigr] \pt$. 
By expressing this as a mixture 
of 
conditional distributions given
$\matches$ and $\pt$, and then using 
Jensen's inequality applied to $x \to x \log x$, 
the entropy can be bounded above by the expected value of
\begin{eqnarray}
\ent\Bigl( \mu
\Bigl[ \prod_{i: m(i) \leq i} c(i, m(i)) \Bigr] \pt \biggiven
\matches,
\pt\Bigr) 
&=&
\ent\Bigl( \mu
\Bigl[ \prod_{i: m(i) \leq i} c(i, m(i)) \Bigr] \biggiven
\matches,
\pt\Bigr) \\
&=&
\label{bigperm}
\ent\Bigl( \mu
\Bigl[ \prod_{i: m(i) \leq i} c(i, m(i)) \Bigr] \biggiven
\matches\Bigr),
\end{eqnarray}
where the first equality holds by Proposition \ref{det}
and the second
equality holds because the permutation $\mu$, the 
product of
collisions
$c(i, m(i))$ and $\pt$ are conditionally independent given $\matches$.  
For $1 \leq k \leq n$, let 
\[
\nu_k =
\prod_{i: m(i) \leq i \leq k} c(i, m(i)).
\]
Note that
the right hand side of (\ref{bigperm})
is $\ent( \mu \nu_n | \matches)$ and $\nu_0 = \id$.  
Since $\mu$ is independent of $\matches$, we have
$\ent(\mu \given \matches) = \ent(\mu)$ and hence
\[
\ent(\mu \nu_n \given \matches) - \ent(\mu) 
= 
\sum_{k=1}^n \ent(\mu \nu_k \given \matches) - \ent(\mu \nu_{k-1} \given \matches).
\]
Thus, it is enough to show that for every $k$ we have
\begin{equation}
\label{ourclaim}
\e\Bigl(
\ent( \mu \nu_k \given \matches) - 
\ent( \mu \nu_{k-1} \given \matches) 
\Bigr)
\leq {-CA_k E_k \over \log n}. 
\end{equation}
Note that if $m(k) > k$ then $\nu_k = \nu_{k-1}$. If $m(k) \leq k$
then
$\nu_k = 
\nu_{k-1} \, c(k, m(k)$).
We can now proceed in a way that is analogous to the proof of Lemma 
\ref{mainlemma}.
Note that
\[
\mu \nu_k 
= \half \mu \nu_{k-1} + \half \mu \nu_{k-1} (k, m(k)).
\]

Fix $i \leq k$, 
let $\lambda  = \mu \nu_{k-1}$ and let
$\lt = \lambda (k,i)$. Note that $\lt$ and $\lambda$ are the same, 
except that
$\lt^{-1}(k) = \lambda^{-1}(i)$ and $\lambda^{-1}(k) = \lt^{-1}(i)$.
Note also that $\nu_{k-1}$ has
$k+1, \dots, n$ as fixed points, so
$(\lambda^{-1}(k+1),\dots, 
\lambda^{-1}(n)) =
(\mu^{-1}(k+1),\dots, 
\mu^{-1}(n))$.
Let 
\begin{eqnarray*}
\f_{k+1} &=& \sigma( \mu^{-1}(k+1), \dots, \mu^{-1}(n)) \\
&=& \sigma( \lambda^{-1}(k+1), \dots, \lambda^{-1}(n)), \\
\end{eqnarray*}
and define
$\fkp = \sigma(\f_{k+1}, \matches)$. Then
we have $\enthat( 
\lt \given \fkp) = \enthat(\lambda \given \fkp)$ 
and hence
$$
\enthat( 
\lambda
c(k,i) \given \fkp
) - \ent(\lambda \given \fkp) = 
-d(
\law(
\lt \given \fkp),
\law(\lambda \given \fkp)).$$ 
But by the projection lemma, 
\begin{eqnarray*}
d\Bigl(
\law(
\lt \given \fkp),
\law(\lambda \given \fkp)\Bigr) &\geq&
d\Bigl(\law(\lt^{-1}(k) \given \fkp),
\law(\lambda^{-1}(k) \given \fkp)\Bigr) \\
&=&
d\Bigl( \law(\lambda^{-1}(i) | \fkp), \law(\lambda^{-1}(k) \given \fkp)\Bigr).
\end{eqnarray*}
Thus, since $m(k)$ is $\fkp$-measurable, on the event that $m(k) \leq
k$
we have 
\begin{eqnarray*}
\label{ourclaim2}
\ent( \mu \nu_k \given \fkp) - 
\ent( \mu \nu_{k-1} \given \fkp)
&=&
\enthat( 
\lambda
c(k,m(k)) \given \fkp
) - \ent(\lambda \given \fkp) \\
&\leq&
-d\Bigl( \law(\lambda^{-1}(m(k)) | \fkp), \law(\lambda^{-1}(k) 
\given \fkp)\Bigr) \\
&=& - \sum_{i \leq k} \one(m(k) = i) 
d\Bigl(\law(\mu^{-1}(i) \given \f_{k+1}), 
\law(\mu^{-1}(k) \given \f_{k+1})\Bigr),
\end{eqnarray*}
where in the third line 
we replaced $\lambda$ by $\mu$ 
because $\nu_{k-1}$ does not contain the collision $c(k, m(k))$ and
hence
has $k$ and $m(k)$ as fixed points, and 
we replaced the sigma field $\fkp$ by $\f_{k+1}$ because $\mu$ is
independent of $\matches$. 
Taking expectations gives
\begin{eqnarray}
\nonumber
\e\Bigl(
\ent( \mu \nu_k \given \fkp) - 
\ent( \mu \nu_{k-1} \given \fkp) 
\Bigr)
&\leq&
- \e\Bigl( 
\sum_{i \leq k} \P(m(k) = i) d\Bigl(\law(\mu^{-1}(i) \given \f_{k+1}), 
\law(\mu^{-1}(k) \given \f_{k+1})\Bigr)\Bigr) \\
\nonumber
&\leq&
- \e\Bigl( A_k k^{-1} 
\sum_{i \leq k} d\Bigl(\law(\mu^{-1}(i) \given \f_{k+1}), 
\law(\mu^{-1}(k) \given \f_{k+1})\Bigr) \Bigr) \\
&\leq&
\label{dstuff}
- \e \Bigl(A_k d\Bigl( k^{-1} \sum_{i \leq k} \law(\mu^{-1}(i) 
\given \f_{k+1}), 
\law(\mu^{-1}(k) \given \f_{k+1}) \Bigr) \Bigr).
\end{eqnarray}
where the second inequality follows by the $A_k$-uniformity of 
$m(k)$ 
and the independence of $m(k)$ and $\mu$, 
and the third inequality
is by Proposition \ref{conv}.
The first argument of $d(\cdot, \, \cdot)$ 
in the right hand side of equation
(\ref{dstuff}) is the uniform distribution over 
$\{1, \dots, n\} - \{\mu^{-1}(k+1), \dots, 
\mu^{-1}(n)\}$. Thus the right hand side of (\ref{dstuff}) 
is 
\begin{eqnarray}
& & - A_k \e \Bigl( d\Bigl(\u, \law(\mu^{-1}(k) \given \f_{k+1})
\Bigr) \Bigr)\\
\label{diff2}
&\leq& 
- {C A_k \over \log n} \e(\ent( \law(\mu^{-1}(k) \given \f_{k+1}))
=
- {C A_k E_k \over \log n},
\end{eqnarray}
where the inequality holds
by Lemma \ref{dent}. Since $\mu \nu_{k}$ and
$\mu \nu_{k-1}$ agree in positions
$k+1, \dots, n$, the portion of their respective entropies 
that is attributable to those positions
coincides, hence
Proposition
\ref{decomp} 
and equation (\ref{diff2}) yield the theorem.
\end{proof}
\begin{remark}
Since for any distribution $p$ we have
$d(p, p) = 0$, equation (\ref{dstuff})
is still true if $m(k)$ is  
only $A_k$--uniform over $\{0, \dots, k-1\}$. So the 
assumptions of the 
theorem 
can be relaxed so that there is no lower bound necessary on the 
probability that $m(k) = k$. 
\end{remark}

\section{Thorp shuffle}
\label{secthorp}
In this section we show that Theorem \ref{maintheorem}
implies an improved bound for the Thorp shuffle. 
Recall that the Thorp shuffle has the following description. 
Assume that the number of cards, $n$, is even.
Cut the deck into two equal piles. 
Drop the first card from the left pile or the
right
pile according to the outcome of a fair coin flip;  
then drop from the other pile. 
Continue this way, 
with independent coin flips deciding whether to drop {\sc left-right} or 
{\sc right-left} each time, until 
both piles are empty. 

We will actually work with the time reversal of the Thorp shuffle,
which clearly has the same mixing time. 
Suppose that we label the positions in the deck
$0, 1, \dots, n-1$. 
Note that the Thorp shuffle can be described in the following way.
Each step,
for $x$ with $0 \leq x \leq {n \over 2} - 1$, 
the cards at positions
$x$ and $x + n/2$ collide
and are moved to positions 
$2x \bmod \, n$ and $2x + 1 \bmod \, n$.
Thus, the time reversal can be described as follows. 
Each step, for even numbers $x \in \{0, \dots, n-2\}$, 
the cards in positions $x$ and $x+1$ collide
and are moved to positions $x/2$ and ${x/2}
+ n/2$. 

We write $\pt$ for a product of $t$ 
i.i.d.~copies of the
reverse Thorp shuffle.
Our main lemma is the following.

\begin{lemma}
\label{tlem}
Let $t = \lc \log_2 n \rc$. 
There is a universal constant $C$ such that
for any random permutation $\mu$ 
we have
\[
\ent(\mu \pt) \leq (1-C/\log^2 n) \ent( \mu).
\]
\end{lemma}
\begin{proof}
Partition the locations $0, \dots, n-1$ into 
intervals $I_m$ as follows.
Let $I_0 = \{0\}$, and 
for $m = 1,2,\dots, \lc \log_2 n \rc,$ define $I_m = \{2^{m-1}, \dots, 2^m-1\}
\cap \{0, \dots, n-1\}$.

For $i \in \{0, \dots, n-1\}$, 
define $E_i = \ent(\mu, i)$. We can write the entropy of $\mu$ as 
\[
\ent(\mu) = \sum_{m} \sum_{i \in I_m} E_i.
\]
Let $\mss$ be the value of $m$ that maximizes $\sum_{i \in m} E_i$. Then
\[
\sum_{j \in I_\mss} E_j \geq {c \over \log n} \ent(\mu),
\]
for a constant $c$. 
Since the reverse Thorp shuffle is in Monte form, we 
may use Theorem \ref{maintheorem}.
We will also use the remark immediately following Theorem
\ref{maintheorem}, which says that the distribution of the card 
matched with $i$ need only be $A_i$ uniform over $\{j: j < i\}$
in order for the conclusions of the theorem to hold.
Fix $m$ with $1 \leq m \leq \lc \log_2 n \rc$.
We will show that the assumptions of the theorem hold with $t = \lc \log_2 n \rc$, 
\[
A_i = 
\left\{\begin{array}{ll}
1/4 
& 
\mbox{if $i \in I_\ms$;}\\
0
& \mbox{otherwise,}\\
\end{array}
\right.
\]
and 
the random variable $T$ defined as follows.   
Let $T$ be any random variable that satisfies
\be
\lab{goodtime}
\P(T = r) \ge 2^{r - \ms - 1}, 
\ee
for $r = 0, \dots, \ms$. \\

Fix $i \in I_\ms$.
We shall show that for any $j < i$ we have 
$\P(m(i) = j) \ge 1/4i$.
Define $f: \Z \to \Z$ by $f(t) = \floor{t/2}$.
Note that if $X_s(j)$ denotes the position
of card $j$ at time $s$, then $X_s(j) = f(X_{s-1}(j)) + Z_s(j)$,
where $Z_s(j)$ is a random  ``offset'' 
whose
distribution is uniform over $\{0, n/2\}$. Note that in step of 
the shuffle, the distance between a pair of cards is cut roughly in 
half if they have the same offsets. More precisely,
if $x > y$ then 
\be
\lab{closer}
f(x) - f(y)
\leq
\left\{\begin{array}{ll}
(x-y)/2 & 
\mbox{if $x$ is odd or $y$ is even;}\\
(x-y)/2 + \half & \mbox{otherwise.}\\
\end{array}
\right.
\ee
It follows that
$\ceiling{ \log_2(f(x) - f(y))} \leq 
\ceiling{ (\log_2 (x - y))}$
and 
$\ceiling{ \log_2(f(x) - f(y))} \leq 
\ceiling{ (\log_2 (x - y))} -1$
unless $x = y+1$ and $x$ is even. 

Say that two positions $x$ and $y$ 
are neighbors if $|x - y| = 1$ and 
$\min(x,y)$ is even. (Note that in each step
of the reverse Thorp shuffle,  the neighbors collide.) 
Since $n$ is even we can write 
$n/2$ = $2^k l$ for some $k \geq 0$ and odd integer $l$. 
Fix $i$ and $j$ with $j \leq i$. 
We shall show that $\P(m(i) = j) \geq 1/4i$.

First, we claim that $\p(\mbox{$X_m(j)$ is even}) \geq \half$.  
To see this, note that $f^\ms(j) = 0$, 
where we write $f^r$ for the $r$-fold iterate of $f$.
Hence, if $m \leq k$, then $X_\ms(j) = \sum_{r=0}^{\ms-1} 2^{-r}
Z_{\ms-r}(j)$. Each of the $Z_{\ms-r}(j)$ is either $0$ or 
$2^k l$, so each term in the sum is even.
Assume now that
$\ms > k$. 
Suppose that the value of $Z_{\ms-k}(j)$ (which is either 
$0$ or $n/2$) is determined by an unbiased coin flip. For $\ms - k \leq s \leq \ms$, 
let $X'_s(j)$ be what the position of card $j$ at time $s$ would have been 
if the outcome of the coin flip determining $Z_{\ms-k}$ had been 
different. Since $f(x) - f(y) = \half(x-y)$ if $x - y$ is even, 
it follows that $| X'_s(j) - X_s(j) | = 2^{\ms-s}l$
for $\ms-k \leq s \leq \ms$. Thus 
$| X'_\ms(j) - X_\ms(j) | = l$, which is odd. So one of 
$X'_\ms(j)$ and $X_\ms(j)$ is odd and the other is even.
Since they have the same distribution, they are each even with 
probability $\half$. 

  Let $y_0 = X_0(i)$, and for $s \geq 1$ let $y_s = f(y_{s-1}) +
  Z_s(j)$,
i.e., where card $i$ would be located after $s$ steps if its offsets were the
  same
as those for $j$. 
Let $\tau = \min\{s: |y_s - X_s(j)| = 1$ and $X_s(j)$ is even$\}$. 
Since $|i - j| \leq 2^\ms$ equation (\ref{closer}) 
and the sentence immediately following it
imply that
there must be a value of $s \leq \ms$ such that
$|y_s - X_s(j)| = 1$. Combining this with the fact that 
$X_\ms(j)$ is even with probability at least $\half$ gives
$\P(\tau \leq \ms) \geq \half$. Furthermore, 
given $\tau = r$, the conditional probability that $X_s(i) = y_s$ for 
$0 \leq s \leq r$ (and hence $i$ and $j$ collide at time $\tau$) 
is $2^{-r}$. Finally,
since
assumption (\ref{goodtime}) gives
$\P(T = r) \geq 2^{r - \ms - 1}$, 
It follows that $\P(m(i) = j) \geq 
2^{-\ms-2} \geq {1 \over 4i}$. 

  We have shown that the assumptions of Theorem 
(\ref{maintheorem}) are met with $t = \lc \log_2 n \rc$ and $A_i = 1/4$
for $i \in I_\ms$.
Applying this with $m = \mss$
shows that
for any permutation $\mu$, 
we have $\ent( \mu \pt) \leq (1 - C/\log^2 n) \ent(\mu)$, 
for a universal constant $C$.
It follows that 
for any $B \in \{1 ,2, \dots\}$ we have
\begin{eqnarray*}
\ent( \pi_{(B t \log^3 n)}) &\leq& (1 - C/\log^2 n)^{B \log^3 n} \,\ent(\id)  \\
&\leq&
n^{1 - CB} \log n,
\end{eqnarray*} 
since $\ent(\id) = \log n! \leq n \log n$ and $1-u \leq  e^{-u}$ for all $u$.
If $B$ is large enough so that $n^{1-CB} \log n \leq \eighth$ for all $n$, then
$\ent( \pi_{(Bt \log^3 n)}) \leq \eighth$ and hence
$|| \pi_{(Bt \log^3 n)} - \u || \leq \quarter$ by equation (\ref{totent}).
It follows that the mixing time is at most $Bt \log^3 n = O(\log^4 n)$.
\end{proof}

\section{$L$-reversal chain}
\label{seclrev}
In this section we analyze Durrett's $L$-reversal chain. 
Recall that the $L$-reversal chain has two parameters, $n$ and $L$.
The cards are located at the
vertices of an $n$-cycle, which we label
$\{0, \dots, n -1\}$. Each step, a vertex $v$ 
and a number $l \in \{0, \dots, L\}$
are
chosen independently and uniformly at random.
Then the interval of cards $v, v+1, \dots, v + l$ 
is reversed, where the numbers are taken mod $n$.
Equivalently, each step a (nonempty) interval of length 
at most $L$ (i.e., of size between $1$ and $L+1$) is 
chosen uniformly at random and reversed. 
We shall assume that $L > L_0$ 
for a suitable value of $L_0$ and $n \geq 4L$.
The cases where $L$ is constant and 
where $n \leq cL$ for a constant $c$ were both treated in 
\cite{durrett}.

We put the shuffle in Monte form as follows.
Let $\mu_{i,j}$ denote the permutation that reverses the cards in
positions $i, i+1, \dots, j$ and leaves the rest unchanged.
Let $Z$ be uniform over $\{1, \dots, L\}$. 
Choose $v$ uniformly at random from $\{0, \dots, n-1\}$
and let
\be
\lab{putinmonte}
\pi
=
\left\{\begin{array}{ll}
\mu_{v, v+L} & 
\mbox{with probability ${1 \over 2(L+1)}$;}\\
\mu_{v, v+L-1} & 
\mbox{with probability ${1 \over 2(L+1)}$;}\\
\mu_{v, v+Z} c(v, Z) & 
\mbox{with probability ${L \over (L+1)}$.}\\
\end{array}
\right.
\ee
Since $\mu_{v, v+2}(v, v+2) = \id$
and $\mu_{v, v+1}(v, v+1) = \id$,
it is easily verified that
$\pi$ has the distribution of an $L$-reversal shuffle.

We
write
$\pt$ for a product of $t$ i.i.d.~copies of the
$L$-reversal shuffle. 
Our main technical lemma is the following.
\begin{lemma}
\label{tl}
There is a universal constant $C$ such that
for any random permutation $\mu$ 
there is a value of  
$t \in \{1, \dots, 
{Cn^3 \over L^3}\}$
such that 
\[
\ent(\mu \pt) \leq (1-f(t)) \ent( \mu),
\]
where $f(t) = {\gamma \over \log^2 n} \Bigl( {t \over n} \wedge 1\Bigr)$,
for a universal constant $\gamma$. 
\end{lemma}
Before proving Lemma \ref{tl}, we first show how 
it gives the claimed mixing time
bound.
\begin{lemma}
\label{mixingtime}
The mixing time for the $L$-reversal chain is 
$O\Bigl( 
(n \vee {n^3 \over L^3})\log^3n  \Bigr)$.
\end{lemma}
\begin{proof}
Let $t$ and $f$ be as defined in Lemma \ref{tl}. Then 
\be
\lab{rell}
{t \over f(t)} = \gamma^{-1} (\log^2 n)  t \Bigl({n \over t} \vee 1 \Bigr)
= \gamma^{-1} \log^2 n ({n \vee t}) \leq T,
\ee
where $T = \gamma \log^2 n [ n \vee {Cn^3 \over L^3}]$.
Note that $1/T$ is a bound on the long run rate of entropy loss
per unit of time.
Lemma \ref{tl} implies that
there is a 
$t_1 \in \{1, \dots, 
{Cn^3 \over L^3}\}$
such that 
\[
\ent(\pi_1 \cdots \pi_{t_1}) \leq (1-f(t_1)) \ent( \id),
\]
and a 
$t_2 \in \{1, \dots, 
{Cn^3 \over L^3}\}$
such that 
\[
\ent(\pi_1 \cdots \pi_{t_1 + t_2}) \leq (1-f(t_2)) 
\ent( \pi_1 \cdots \pi_{t_1}),
\]
etc. Continue this way to define $t_3, t_4$, and so on. 
For $j \geq 1$ let $\tau_j = \sum_{i=1}^j t_i$. 
Then
\begin{eqnarray}
\ent( \pi_{(\tau_j)}) &\leq& \Bigl[\prod_{i=1}^j (1 - f(t_j)) \Bigr] 
\ent(\id)\\
&\leq& \exp\Bigl( - \sum_{i=1}^{j} f(t_j) \Bigr) \ent(\id).
\end{eqnarray}
But since $t_j \leq T f(t_j)$ by equation (\ref{rell}), we have
\[
\tau_j = \sum_{i=1}^{j} t_j
\leq T \sum_{i=0}^{j} f(t_j).
\]
It follows that
\begin{eqnarray}
\ent( \pi_{(\tau_j)}) &\leq& 
\exp\Bigl({-\tau_j \over T} \Bigr) \ent(\id).
\end{eqnarray}
Since $\ent(\id) = \log n! \leq  n\log n$, it follows that
if $\tau_j \geq 
T \log(8n \log n)$ 
we have $\ent(\pi_{(\tau_j)}) \leq \eighth$
and hence $|| \pi_{(\tau_j)} - \u || \leq \quarter$
by equation (\ref{totent}). It follows that 
the mixing time is $O(T \log(8n \log n)) = 
O\Bigl( 
(n \vee {n^3 \over L^3})\log^3n  \Bigr)$.
\end{proof}
We shall now prove Lemma \ref{tl}.
\begin{proofof}{Proof of Lemma \ref{tl}}
Let $m = \lceil \log_2 (n/L) \rceil$. Then we can partition  the 
set of locations $\{0, \dots, n-1\}$ into $m+1$ intervals as 
follows. Let $I_0 = \{0, \dots, L\}$, and
for $1 \leq k \leq m$ define 
$I_k = \{2^{k-1} L + 1, \dots, 2^k L\}
\cap \{0, \dots, n-1\}$.  
Define $E_k = \e(\ent(\mu, k))$. 
Note that we can write the entropy of $\mu$ as 
\be
\lab{eee1}
\ent(\mu) = \sum_{k=0}^m \sum_{j \in I_k} E_j\, .
\ee 
Thus, if $\kss$ maximizes $\sum_{j \in I_k} E_j$, then
\[
\sum_{j \in I_\kss} E_j \geq {1 \over m+1} \ent(\mu).
\]
Suppose first that $\kss=0$. Then we can take $t=1$. 
Let $\pi$ be a random permutation corresponding to one move of the
$L$-reversal chain. 
Let $E$ be the event that $\pi$ 
reverses $a, a+1, \dots, b$ 
for $a,b \in \{0, \dots, L\}$.
Then (using an abuse of notation similar to that in 
Section  \ref{abuse})
we can write $\pi$ as
\[
\pi = \alpha \pi_1 + (1 - \alpha) \pi_2,
\]
where $\alpha = \P(E)$, $\pi_1$ is 
$\pi$ conditioned on $E$, 
and  $\pi_2$ is $\pi$ conditioned on $E^c$. 
Then $\mu \pi = \alpha \mu \pi_1  + (1 - \alpha) \mu \pi_2$ and hence
\begin{eqnarray}
\ent(\mu \pi) 
&=& \ent (
\alpha \mu \pi_1 + (1 - \alpha) \mu \pi_2) \\
&\leq& \alpha \ent( \mu \pi_1) + (1- \alpha) \ent(\mu \pi_2) \\
&\leq& \alpha \ent( \mu \pi_1) + (1- \alpha) \ent(\mu),
\end{eqnarray}
where both inequalities follow from the convexity of $x \to x \log x$.
It follows that
\be
\lab{ebou}
\ent( \mu \pi) - \ent(\mu) \leq \alpha \Bigl[
\ent(\mu \pi_1) - \ent( \mu) \Bigr].
\ee

Note that $\pi_1$
does not move any of the cards in  
locations $\{L+1, \dots, n\}$.
Hence by Proposition \ref{decomp}, the entropy difference 
$\ent(\mu \pi_1) - \ent(\mu)$ is the expected loss in entropy
attributable to positions $\{0, \dots, L\}$, i.e., 
$\e\Bigl( \ent( \mu \pi_1 \given \f_{L+1}) 
- \ent(\mu \given \f_{L+1}) \Bigr)$, where 
$\f_{L+1} = \sigma( \mu^{-1}(L+1), \dots, \mu^{-1}(n-1))$. 
The permutation $\pi_1$
is a step
of a modified $L$-reversal chain 
on the $L  + 1$ cards in 
the line graph $\{0, \dots, L$\},
reversing an interval 
of the form $a, a+1, \dots, b$ for $0 \leq a \leq b \leq L$.   

  In Theorem 6 of \cite{durrett}, it is shown 
(by comparison with shuffling through random transpositions
\cite{ds}; 
see \cite{dsc} for background on comparison techniques)
that the log Sobolev
constant for the $L$-reversal chain on $n$ cards
is at most $B {n^3 \over L^2} \log n$ for a constant $B$. 
This remains true if 
we consider the modified $L$-reversal process on the line graph.
Thus $\pi_1$ has a log Sobolev constant that is at most
$2 B {L} \log L$, and hence 
(by the well-known relationship between the log Sobolev constant and decay of 
relative entropy; see, e.g., \cite{miclo})
multiplying $\mu$ by $\pi_1$ reduces the relative entropy
by at least $1/B' L \log L$ times the entropy attributable to 
positions $\{0, \dots, L\}$, for a constant $B'$. Thus
the right hand side of (\ref{ebou}) is at most
\begin{eqnarray} 
-\alpha(B'L \log L)^{-1} \sum_{j \in I_1} E_j 
&\leq& - (8B'n \log L)^{-1} \sum_{j \in I_1} E_j \\
&=&
- (8B'n \log^2 n)^{-1} \ent(\mu),
\end{eqnarray}
where the second line follows from the fact that $\alpha \geq {L \over
  8n}$.
 
 Next we shall consider the case where $\kss \geq 1$, 
so that the interval is of the form
$\{2^{k-1}L + 1, \dots, 2^k L\} \cap \{0,1 \dots, n-1\}$.
We will use Theorem \ref{maintheorem} to get a decay of entropy in
this case. 
We make the following claim.
\begin{claim}
\label{lclaim} 
Fix $k \geq 1$. There are universal constants $C$ and $\alpha > 0$ such that 
if $t = 
{4^k C n / L^3}$,
$T = t/2$
and 
\[
A_y = 
\left\{\begin{array}{ll}
\alpha ({t \over n} \wedge 1) 
& 
\mbox{if $y \in I_k$;}\\
0
& \mbox{otherwise,}\\
\end{array}
\right.
\]
then the assumptions of Theorem \ref{maintheorem} are satisfied 
by $t, T,$ and the $A_y$.
\end{claim}

In order to prove this claim, it is helpful to know that the $L$-reversal chain enjoys  
certain monotonicity properties. Roughly speaking, the closer two cards are together, 
the more likely they are to collide after a given number of steps. 
Before proving Claim \ref{lclaim}, we shall verify these monotonicity
properties. \\
\\
{\bf Two types of monotonocity.}
Fix $x$ and $y$ in $\{0, \dots, n\}$ 
and let $x_m$ and $y_m$ denote
the positions of cards $x$ and $y$, respectively, at time $m$. 
Define $Z_m = |x_m - y_m|$, i.e., the graph distance 
between $x_m$ and $y_m$ in the $n$-cycle. Note that 
$Z_m$ is a Markov chain. 
We shall need the following lemma.
\begin{lemma}
\label{monlem1}
Let $\phat$ denote the transition matrix of $Z_m$. 
Then $\phat$ is monotone, i.e.,
if $b \geq a$ then $\phat(b, \cdot) \sgeq \phat(a, \cdot)$, 
where $\sgeq$ denotes stochastic domination. 
\end{lemma}
\begin{proof}
Fix positions $u$ and $a$ with $a \leq n/2$, and let $N(a, u)$ 
denote the number of legal intervals (i.e., intervals
of length at most $L$) that move the card in position 
$a$ to position $u$ without moving the card in position $0$. Then
\[
N(a,u) = 
\left\{\begin{array}{ll}
\min(u, \lfloor \half(L - a + u) + 1 \rf)
& \mbox{if $u < a$;}\\
\min(a, \lfloor \half(L - u + a) + 1 \rfloor)
& \mbox{if $u > a$.}\\
\end{array}
\right.
\]
(Recall that we assume that $n \geq 4L$.)
Suppose that $|x_m - y_m| = a$. 
For $u \leq n/2$, let $M(a,u)$ denote the number of legal intervals whose reversal 
at time $m$
would 
make $|x_{m+1} - y_{m+1}| = u$. If $a \neq u$ then $M(a,u)$ counts 
intervals that move $x$ but not $y$ and 
intervals that move $y$ but not $x$. Thus we have 
$M(a,u) = 2(N(a,u) + N(a,n-u))$.
It is easily verified that $M(a, u)$ is nonincreasing in $a$ for $u < a \leq n/2$ 
and nondecreasing in $a$ for $0 < a < u$. It follows that $Z_m$ is monotone. 
\end{proof}

We now prove that $Z_m$ has another type of monotonicity
property. 
Note that in each move of the $L$-reversal process, there are
exactly four cards that are adjacent to a different pair of cards
after the move than they were before. We 
say that those cards are {\it cut} and write, e.g.,  
 ``card $i$ is cut at time $m$''. We say that a location
 is cut if the card in that location is cut. \\
\\
{\bf The cut-stopped process.}
It will be convenient to consider a
modified version $Z'_m$
of 
$Z_m$,  where we introduce two 
absorbing states
$\sgood$ and $\sbad$, and have the following occur when
either $x$ or $y$ is cut. If $x$ and $y$
are within a distance $L$ of each other, then $Z'_m$ transitions to 
$\sgood$; otherwise, it transitions to $\sbad$. 

We shall call this modified process the {\it cut-stopped process.}
We can impose an order on the
state space of $\{Z'_m: m \geq 0\}$ based on the order of the 
positive integers, 
with the additional states $\wgood$ and $\wbad$ as the minimum 
and maximum states, respectively. 

Our next lemma 
says 
that 
the cut-stopped process
$Z'_m$ is monotone with respect to this order.
\begin{lemma}
\label{undom}
The cut-stopped process is monotone. 
\end{lemma}
\begin{proof}   
The proof is a slight modification of the proof of Lemma
\ref{monlem1}. 
Suppose that $Z'_m = z$. 
Note that the probability of absorbing in $0$ in the next step is a nonincreasing
function of $z$, and the probability of absorbing in $\infty$ in the next step is 
a nondecreasing function of $z$. 
The rest of the 
argument is almost identical to the proof of 
Lemma \ref{monlem1}. 
Fix positions $u$ and $a$ with $a \leq n/2$, and let $N'(a, u)$ 
denote the number of intervals 
of length at most $L$ that 
move the card in position 
$a$ to position $u$, but neither move the card in position $0$, 
cut position $0$, nor cut position $a$. Then
\[
N'(a,u) = 
\left\{\begin{array}{ll}
\min(0, u-2, \lfloor \half(L - a + u) \rf)
& 
\mbox{if $u < a$;}\\
\min(0, a-2, \lfloor \half(L - u + a)\rfloor)
& \mbox{if $u < a$.}\\
\end{array}
\right.
\]
Suppose that $|x_m - y_m| = a$. 
For $u \leq n/2$, let $M'(a,u)$ denote the number of legal intervals 
that don't cut $x$ or $y$ and whose reversal 
at time $m$
would 
make $|x_{m+1} - y_{m+1}| = u$. If $a \neq u$ then 
$M'(a,u) = 2(N'(a,u) + N'(a,n-u))$.
It is easily verified that $M'(a, u)$ is nonincreasing in $a$ for $u < a \leq n/2$ 
and nondecreasing in $a$ for $0 < a < u$. It follows that $Z'_m$ is monotone. 
\end{proof}
\noindent
We are now ready to prove Claim \ref{lclaim}.
For the convenience of the reader, we state the claim again.
Recall that 
$I_k = \{2^{k-1} L + 1, \dots, 2^k L\}
\cap \{0, \dots, n-1\}$.  
\\
\\
{\bf Claim \ref{lclaim}} 
{\it There are universal constants $C$ and $\alpha>0$ such that 
if $t = 
{4^k C n / L^3}$,
$T = t/2$
and 
\[
A_y = 
\left\{\begin{array}{ll}
\alpha ({t \over n} \wedge 1) 
& 
\mbox{if $y \in I_k$;}\\
0
& \mbox{otherwise,}\\
\end{array}
\right.
\]
then the assumptions of Theorem \ref{maintheorem} are satisfied 
by $t, T,$ and the $A_y$.}
\begin{proof}
Let $y \in I_k$. We need to show that if $x \leq y$, then with
probability at least $A_y$, 
cards $x$ and $y$ collide between time $T$ and time $t$,
and this is the first collision that either is involved in after time 
$T$. 

Fix $y \in I_k$ and $x$ with $x < y$.  
Let $\tau$ be the first time after time $T$ that either $x$ or $y$
is cut. 
Note that if $x$ and $y$ collide at time $\tau$ and $\tau \leq t$ then $m(x) = y$. 
Thus, given that $|x_\tau - y_\tau| \leq L$ and $\tau \leq t$ 
the conditional 
probability that $m(x) = y$ is at least $1/8L$. 
This is because the number of intervals that cut either $x$
or $y$ is at most $4L$, so the conditional probability that $x$ and $y$ are 
at the endpoints of the interval that is reversed at time $\tau$
is at least $1/4L$. The conditional probability that $x$ and $y$
collide
is at least half of this. 

Thus it is enough
to show that
for a universal constant $\alpha$ we have
\begin{equation}
\label{dagger}
\P(|x_\tau - y_\tau| \leq L, \tau \leq t) \geq \alpha \A L/y.
\end{equation}
For $m \geq 0$ let $Z_m = |x_m - y_m|$. 
Let $\beta > 0$ be a 
constant and suppose that $L > 2\beta$.
We claim that with probability bounded away from $0$ we have 
$|Z_m| \leq \beta L$ for some $m < T$.
To see this, let $M = \min\{m: Z_m \leq L\}$.
First, we will show that with probability bounded away from 
zero we have $M \leq T'$, where $T' = T/2$.
Suppose that $Z_0 > L$. Let $X$ be a random variable with the
distribution of $Z_1 - Z_0$ and let $X_1, X_2, \dots$ be i.i.d.~copies
of $X$. Note that 
the random variable $Z_{T'} - Z_0$ can be coupled with 
the $X_i$ in such a way that
$Z_{T'} - Z_0 \leq \sum_{1=1}^{T'} X_i$
on the event that $M >
T'$. It follows that $\P(M \leq T') \geq \P( \sum_{i \leq T'} X_i \geq
-Z_0)$.
But since when $X$ 
is nonzero (which happens with probability on the order of $L/n$)
it has a typical value on the order of $L$, it has second and third
moments satisfying 
$\sigma^2 \geq 
{C_2L^3/n}$ and 
$\rho \leq {C_3 L^4/n}$,
respectively. 
Berry Esseen bounds (see, e.g.,
\cite{durrettbook}) imply that for a universal constant 
$C_{B}$ we have
\be
\lab{be}
| F_{T'}(x) - \Phi(x)| \leq {C_{B} \rho \over \sigma^3 \sqrt{T'}}
\leq {C' L \over C 
y},
\ee
where $F_{T'}$ is the cumulative distribution function (cdf) of 
${1 \over \sigma \sqrt{T'}} \sum_{i \leq
  T'} X_i$, $\Phi$ is the standard normal cdf, $C'$ is a constant that 
incorporates $C_2, C_3$ and $C_b$, and 
$C$ is the constant appearing in the definition of $t$. 
For the final inequality  we use the fact that $t=4T'$ is within constant 
factors of $C y^2 n/L^3$, since $y \in I_k$. 

Since $y
\geq L$, the quantity (\ref{be}) can be made arbitrarily close to
zero for sufficiently large $C$. It follows that 
$\sum_{i \leq T'} X_i$ is roughly normal with standard deviation 
a large constant times $y$, hence is less than $-Z_0$ with probability bounded
away from zero. 
(Recall that $Z_0 = y - x \leq y$). 
It follows that with probability bounded away from
zero we have $Z_m \leq L$ 
for some $m \leq T/2$. Now note that if $x$ and $y$ are within
distance $L$ then given that one of them moves in the next step, the 
conditional probability that they are brought to within a distance
$\beta L$ is bounded away from zero. 
Since $t$ is much larger than $n/L$, 
there is probability bounded away from zero 
that either $x$ or $y$ is moved between time $m$ and $m + T/2$. This
verifies the claim.

The above claim and the strong Markov property
imply that in order to show (\ref{dagger}),
it is enough to show that 
if 
$|i - j| \leq \beta L$, 
$m' \leq T/2$ and  $\tau$ is the first time that $i$ or $j$ is cut 
after time $m'$,
then for a universal constant $\alpha > 0$ we have
$\P(|i_\tau - j_\tau| \leq L, \tau \leq m' + t/2) \geq \alpha \A L/y$.

For every pair of cards $i$ and $j$, 
let $T(i,j)$ be the first time that either $i$ or $j$ 
is cut after 
time $m'$. 
Define $t' = \min(t/2, n)$. 
Let $A(i,j)$ be the event that
$T(i,j) \leq m' + t'$
and 
at
time $T(i,j)$ 
the distance between $i$ and $j$ is most $L$. 
Let
$f(i,j) = \P(A(i,j))$. 
Since $t' \leq t/2$, 
it is enough to prove that
if $|i - j| \leq \beta L$ then
\be
\lab{rtp}
f(i,j) \geq \alpha \A L/y.
\ee

Since the probability that either $i$ or $j$ 
is involved in a cut on any given step 
is at most $8/n$, we have 
\begin{equation}
\label{gamgam}
f(i,j) \leq \min(1, 8t'/n).
\end{equation} 
Also, note that
\begin{eqnarray*}
\sum_{i,j} f(i,j) &=&
\sum_{l=m'+1}^{m' + t'} \,\, 
\sum_{i,j} \P(\mbox{
$T(i,j) \geq l$, $|i_{l} - j_{l}| \leq L$, either
$i$ or $j$ is cut at time $l$}) \\
&=&
\sum_{k=1}^{t'}
\sum_{\twosubs{u < v}{|u - v| \leq L}} g(u,v,k),
\end{eqnarray*}
where 
$g(u,v,k)$ is the probability that cards in locations $u$ and $v$ 
are cut at time $m' + k$, but neither had
been cut since time $m'$.
Since the $L$-reversal process is symmetric
it is its own time-reversal. Thus, 
$g(u,v,k)$ is the probability that either location $u$ or $v$ is cut in the first
move, but neither  
the card in location $u$ at time $1$ nor the card in location $v$ at time $1$ is
is cut in the next $k - 1$ moves. This probability  is at least
$\sfrac{1}{n} \Bigl( {n - 8 \over n} \Bigr)^{t'-1}$. Since there are 
$nL$ such pairs $(u, v)$, summing over $u,v$ and $k$ gives
\begin{eqnarray*}
\sum_{i,j} f(i,j) &\geq& 
t' nL 
\frac{1}{n} \Bigl( {n - 8 \over n} \Bigr)^{t'-1} \\
&\geq& c' L t',
\end{eqnarray*}
for a universal constant $c'$,
where the second inequality holds because $t' \leq n$.
It follows that for any $i$ we have
\be
\lab{clt}
\sum_{j} f(i,j) 
= {1 \over n} \sum_{i,j} f(i,j) 
\geq c' L t'/n \,. 
\ee
Let $g(i,j) = \P(A(i,j) \cap B(i,j))$ 
where $B(i,j)$ is the event that
at 
no time before time $T(i,j)$ was the distance between $i$ and $j$ 
greater than $Dy$, where the constant $D$ is to be specified below.
Note that
\begin{eqnarray}
\label{frown}
\sum_j g(i,j) &\geq& \sum_j f(i,j) - 
\P(A(i,j)
\cap B^c(i,j)),
\end{eqnarray}
where 
$B^c(i,j)$ denotes the complement of 
$B(i,j)$.
We claim that
$\sum_{j} g(i,j) 
\geq c L t'/n$
for a universal constant $c$. To see this, 
fix a card $i$ and $k \leq t'$ and say that a card $u$ is 
{\it bad} if
$|i_0 - u_0| \leq L$, and $\max_{0 \leq r \leq m' + k} |i_r - u_r| > Dy$. 
Since the $L$-reversal process is symmetric, 
and the probability that $i$ or $u$ is cut in any given step is at
most
$8/n$, we have
\be
\lab{sstar}
\sum_j \P\Bigl( A(i,j)
\cap B^c(i,j) 
\cap [T(i,j) = m' + k] \Bigr)
\leq {8 \over n} \e(B),
\ee
where $B$ is the number of bad cards.
Let $u$ be a card initially within distance
$L$ of card $i$. If $u_m$ is the position of card 
$u$ at time $m$, then we can write $u_m = u + W_1 + \cdots
W_m \; (\bmod \,\, n)$,
where $W_j \in \{-L, \dots, L\}$ is the displacement of card $u$ at time $j$. 
Define $u'_m = u + W_1 + \cdots
W_m$ (i.e., like $u_m$, but without the $\bmod \,\, n$), with a 
similar definition for $i'_m$. Then $u'_m$ is a
symmetric random walk on the integers. Each step there is a jump with 
probability on the order of $L/n$ and the sizes of jumps are at most
$L$. It follows that for sufficiently large $A$,
the probability that 
$\max_{1 \leq m \leq k} |u'_m - u'| > A  ({kL \over n})^{1/2} L$
can be made arbitrarily close to
zero. Since $k$ is at most a constant times ${y^2n \over L^3}$, 
we have
$A  ({kL \over n})^{1/2} L \leq A' y$ for a constant $A'$. 
A similar argument applies to $i'_m$. Finally, since 
$|i_m - u_m| \leq |i'_m - u'_m|$
 (where the first $| \cdot |$ refers to distance in the $n$-cycle),
it follows that for any $\epsilon > 0$, if $D$ is large enough then 
$\P( 
\max_{1 \leq m \leq k} |i_m - u_m| > Dy) < \epsilon.$
Thus, since there are at most $2L$ cards initially within a distance $L$ of card $i$, we
have $E(B) \leq 2L\epsilon.$ Hence, summing equation (\ref{sstar}) over 
$k \leq t'$ gives
\be
\lab{smiley}
\sum_j \P \Bigl( A(i,j)
\cap B^c(i,j) 
\Bigr)
\leq 16L \epsilon {t'/n} \; . 
\ee
Combining this with equations (\ref{frown})
and (\ref{clt}) gives
\begin{eqnarray}
\label{ssstar}
\sum_j g(i,j) &\geq& 
cL t'/n,
\end{eqnarray}
for a constant $c$, if $\epsilon$ is small enough. We now define $\beta$ to be a constant smaller than $c/32$.
Since for any $j$ we have $g(i,j) \leq f(i,j) \leq 8t'/n$ (by equation (\ref{gamgam})),
we have
$\sum_{j:  |i - j| \leq \beta L} g(i,j) \leq 16 \beta L t'/n \leq cLt'/2n$,
and hence
\[
\sum_{j: |i-j| > \beta L} g(i,j) \geq cLt'/2n,
\]
by equation (\ref{ssstar}). 
Since $g(i,j) = 0$ for $|j-i| > Dy$, the average value of $g(i,j)$, 
where $j$ ranges over values such that $\beta L < |i - j| \leq Dy$, must be at least
${cLt' /4Dyn} \geq {\alpha L \A /y}$, for a constant $\alpha$. 
Since both $Z_m$ and 
the cut-stopped process $Z'_m$ are monotone by Lemmas
\ref{monlem1} 
and
\ref{undom}, the function $g(i,j)$ is 
nonincreasing in $|i-j|$. 
It follows that $g(i,j) \geq
{\alpha L \A /y}$ 
if $|i - j| \leq \beta L$. Since $g \leq f$,
this verifies equation (\ref{rtp}), which 
completes the 
proof of Claim \ref{lclaim}.
\end{proof}
\noindent
Using Claim \ref{lclaim} with $k = \kss$ and applying  
Theorem \ref{maintheorem} gives
\[
\ent( \mu \pt ) - \ent(\mu) \leq {-C \over \log^2 n} \Bigl({t
  \over n} \wedge 1 \Bigr) \ent(\mu),
\]
for a universal constant $C$, 
and the proof of Lemma \ref{tl} is complete.
\end{proofof}

\noindent {\bf Acknowledgments.} I am grateful to  A.~Soshnikov 
for many valuable conversations during the early stages of this work.

\end{document}